\documentclass[12pt]{amsart}
\usepackage[all]{xy}
\usepackage{amsfonts}
\usepackage{amssymb}
\usepackage{pb-diagram}
\usepackage{enumerate}
\usepackage{color}
\usepackage{ulem}

\newtheorem{teo}{Theorem}[section]
\newtheorem{lem}[teo]{Lemma}
\newtheorem{prop}[teo]{Proposition}

\newtheorem{dfn}[teo]{Definition}
\newtheorem{rk}[teo]{Remark}
\newtheorem{ex}[teo]{Example}

\def\<{\langle}
\def\>{\rangle}
\def\ss{\subset}

\def\a{\alpha}

\def\e{\varepsilon}

\def\f{{\varphi}}

\def\Id{\operatorname{Id}}

\def\Index{\operatorname{Index}}

\def\1{\mathbf 1}

\begin{document}

\title{Quantization of branched coverings}

\author{Alexander Pavlov}
\thanks{Partially supported by the RFBR (grant 08-01-00034), by the
joint RFBR-DFG project (RFBR grant 07-01-91555 / DFG project
``K-Theory, $C^*$-algebras, and Index theory''), and by the
`Italian project Cofin06 - Noncommutative Geometry, Quantum Groups
and Applications'}
\address{All-Russian Institute of Scientific and Technical Information,
Russian Academy of Sciences (VINITI RAS), Usievicha str. 20,
125190 Moscow A-190, Russia} \email{axpavlov@gmail.com}

\author{Evgenij Troitsky}
\thanks{}
\address{Dept. of Mech. and Math., Moscow State University,
119992 GSP-2  Moscow, Russia}
\email{troitsky@mech.math.msu.su}
\urladdr{
http://mech.math.msu.su/\~{}troitsky}

\begin{abstract}
We identify branched coverings (continuous open surjections
$p:Y\to X$ of Hausdorff spaces with uniformly bounded number of
pre-images) with Hilbert C*-modules $C(Y)$ over $C(X)$ and with
faithful unital positive conditional expectations $E:C(Y)\to C(X)$
topologically of index-finite type. The case of non-branched
coverings corresponds to projective finitely generated modules and
expectations (algebraically) of index-finite type. This allows to
define non-commutative analogues of (branched) coverings.
\end{abstract}

 \maketitle

%\tableofcontents

\section{Introduction}
The purpose of the present paper is to obtain an
appropriate description of branched coverings in
terms of (commutative) C*-algebras and their modules
in such a way that it admits a natural generalization
to a non-commutative setting. In fact, we will obtain
two (closely related to each other) descriptions.

A \emph{branch covering} (in this paper) is a closed and open
continuous surjection of compact Hausdorff spaces $p:Y\to X$ with
a finite bounded number of pre-images $\# p^{-1}(x)$, $x\in X$.
(In Section \ref{sec:covbrancov} we describe some properties of
branch coverings and their equivalent descriptions.)

The main result of the present paper is the following
theorem.

\begin{teo}\label{teo:mainbranchcov}
Suppose, $i:C(X) \to C(Y)$ is an inclusion, where
$X$ and $Y$ are compact Hausdorff spaces. Let $p=i^*$
be its Gelfand dual surjection $p:Y\to X$.
Then the following properties are equivalent:
\begin{enumerate}[\rm 1)]
    \item The surjection $p$ is a branched covering.
    \item Consider $C(Y)$ as a $C(X)$-module with
        respect to the natural action induced by $i$. Then
$C(Y)$ can be equipped with an inner $C(X)$-product
    in such a way that it becomes a $($complete$)$
    Hilbert $C(X)$-module.
    \item It is possible to define a positive unital
    conditional expectation $E:C(Y)\to C(X)$ topologically
    of  index-finite type $($in the sense of \cite{BDH}$)$.
\end{enumerate}
\end{teo}

\begin{proof}
The implication 1)$\Rightarrow$2) will be proved
in Theorem \ref{teo:branched_coverings_inner_product}.
The implication 3)$\Rightarrow$1) will be proved
in Theorem \ref{teo:branched_cov_cond_expectations}.
The equivalence 2)$\Leftrightarrow$3) is known
(see \cite{FrankKirJOP} and Proposition \ref{prop:FMTZAA}
below).
\end{proof}

This theorem suggests how to quantize branched coverings. More
precisely we can introduce the following definition.

\begin{dfn}\label{dfn:ncbcov}\rm
A \emph{non-commutative branched covering} is
a pair $(B,A)$
consisting of a $C^*$-algebra $B$ and its $C^*$-subalgebra $A$
with common unity, such that one of the following
equivalent (by \cite[Theorem 1]{FrankKirJOP}) conditions
holds.
\begin{enumerate}[\rm 1)]
\item The algebra $B$ may be equipped with an inner $A$-valued
product     in such a way that it becomes a $($complete$)$
    Hilbert $A$-module.
    \item There exists a positive
conditional  expectation $E: B\rightarrow A$ topologically of
index-finite type.
    \end{enumerate}
 \end{dfn}

The above theorem and definition can be specialized
to the case of (non-singular finite-fold) coverings
in the following way.
The most part of the next theorem is known.

\begin{teo}\label{teo:mainhcover}
Suppose, $i:C(X) \to C(Y)$ is an inclusion, where
$X$ and $Y$ are compact Hausdorff spaces. Let $p=i^*$
be its Gelfand dual surjection $p:Y\to X$.
Then the following properties are equivalent:
\begin{enumerate}[\rm 1)]
    \item The surjection $p$ is a finite-fold covering.
    \item  The
    module $C(Y)$ may be equipped with an inner $C(X)$-product
    in such a way that it becomes a finitely generated
    projective Hilbert $C(X)$-module.
    \item It is possible to define a positive unital
    conditional expectation $E:C(Y)\to C(X)$ $($algebraically$)$
    of index-finite type $($in the sense of \cite{Wata}$)$.
\end{enumerate}
\end{teo}

\begin{proof}
The implication 1)$\Rightarrow$3) is proved
in \cite[Proposition 2.8.9]{Wata}.
The implication 2)$\Rightarrow$1) will be proved
in Theorem \ref{teo:criterium_fin_gen_mod2}.
The equivalence 2)$\Leftrightarrow$3)
can be extracted from \cite[pp. 92--93]{Wata}
(see Theorem \ref{teo:cond_expect}
below).
\end{proof}

\begin{dfn}\label{dfn:nccov}\rm
A \emph{non-commutative covering} is a pair $(B,A)$ consisting of
a $C^*$-algebra $B$ and its $C^*$-subalgebra $A$ with common
unity, such that one of the following equivalent (this may be
extracted from \cite[pp. 92--93]{Wata}, see Theorem
\ref{teo:cond_expect} below) conditions holds.
\begin{enumerate}[\rm 1)]
\item The algebra $B$ may be equipped with an inner $A$-valued
product     in such a way that it becomes a finitely generated
projective Hilbert $A$-module.
    \item There exists a positive
conditional  expectation $E: B\rightarrow A$ algebraically of
index-finite type.
    \end{enumerate}
 \end{dfn}

Our research continues the research on spaces and modules arising
from discrete group actions  (cf. \cite{FMTZAA,TroPAMS}). Apart
from the mentioned above papers let us indicate the research of
Buchstaber, Rees, Gugnin and others on algebraic definition and
topological applications of Dold-Smith ramified coverings (see,
e.g., \cite{BR02, BR08, Gug05}).

A number of known as well as of  new facts about branched
coverings are collected in Section \ref{sec:covbrancov}.
Preliminaries on Hilbert modules, basic lemmas and examples are
contained in Section \ref{sec:hilbmodcov}. Section
\ref{sec:branchmodu} deals mostly with Hilbert module aspects of
proofs of the main theorems. At its end a couple of related
statements concerning other types of Hilbert modules is proved.
Section \ref{sec:branchcovexpect} is devoted to conditional
expectations and to the corresponding parts of proofs. At the end
of the section the role of index elements is discussed.

{\bf Acknowledgement:}
The authors are grateful to M.~Frank, V.~M.~Manuilov,
A.~S.~Mish\-che\-n\-ko and T.~Schick for helpful discussions.
In particular, A.~S.~Mishchenko (after a talk of one of us
(ET) on the results of \cite{TroPAMS})
has posed some of the problems solved in the present paper.

%%%%%%%%%%%%%%%%%%%%%%%%%%%%%
\section{Branched coverings}\label{sec:covbrancov}
%%%%%%%%%%%%%%%%%%%%%%%%%%%%%

In this section we present (mostly known) statements about
continuous surjections of Hausdorff spaces.
Let
\begin{equation}\label{eq:p_Y_X}
p : Y\rightarrow X
\end{equation}
be a continuous surjection of compact Hausdorff spaces,
in particular, a closed map.

\begin{dfn}\label{dfn:regular_neighborhood}
{\rm Let us consider the map (\ref{eq:p_Y_X}) and a certain point
$x$ of $X$, which has a finite number of pre-images $y_1,
\dots,y_m$. Then a neighborhood $U$ of $x$ is said to be
\emph{regular} if
\begin{gather}\label{eq:neiborhoods}
p^{-1}(U)= V_1\sqcup\dots\sqcup V_m,
\end{gather}
where $V_i$ are some neighborhoods of $y_i$, $i=1,\dots, m$.}
\end{dfn}

\begin{lem}\label{lem:neiborhoods}
Let $p : Y\rightarrow X$ be a continuous closed map of Hausdorff
spaces. Then any point $x$ of $X$ with a finite number of
pre-images has a regular neighborhood.
\end{lem}

\begin{proof} Suppose the pre-image of $x$ consists of points $y_1,\dots,y_m$.
These points can be separated by pairwise disjoint neighborhoods
$V'_1, \ldots ,V'_m$. Then the set $U=X\setminus
p(Y\setminus\sqcup_{i=1}^m V'_i)$ is an open neighborhood of $x$,
because the map $p$ is closed. Now one can set $V_i=p^{-1}(U)\cap
V_i'$.
\end{proof}

\begin{lem}\label{lem:neiborhoods2}
Let $p : Y\rightarrow X$ be a continuous closed map of Hausdorff
spaces. Suppose $U$ is a regular neighborhood of a point $x\in X$
satisfying {\rm(\ref{eq:neiborhoods})} and $U'$ is an open set
satisfying the condition: $x\in U'\subset\overline{U'}\subset U$.
Put $V_i'=p^{-1}(U')\cap V_i$. Then
\begin{enumerate}
    \item $p^{-1}(U')=\sqcup _{i=1}^mV_i';$
    \item $\overline{V_i'}=p^{-1}(\overline{U'})\cap V_i;$
    \item $p^{-1}(\overline{U'})=\sqcup _{i=1}^m\overline{V_i'}.$
\end{enumerate}
\end{lem}

\begin{proof} The first statement is obvious. The second is
true because the map $p$ is closed. And the third immediately
follows from the second.
\end{proof}

Given the covering (\ref{eq:p_Y_X}). Denote by $X_j$ the subset
(stratum) of
$X$ consisting of points that have exactly $j$ pre-images and
reserve the notation $\widehat{X}_j$ for the union
$\bigcup_{i=0}^j X_i$, $j\ge 0$.

 Now for any point $x$ of $X$ consider the collection
(\ref{eq:neiborhoods}) of neighborhoods $U, V_1,\dots, V_m$, where
$m=\#p^{-1}(x).$ Then
\begin{gather*}
    X(x,U)_j^{(k)}=\{z\in U : \#p^{-1}(z)_k=j\},
\end{gather*}
where $p^{-1}(z)_k=p^{-1}(z)\cap V_k$, and
$\widehat{X}(x,U)_j^{(k)}$ stands for the union $\bigcup_{i=0}^j
X(x,U)_i^{(k)}$.

\begin{dfn}\label{dfn:branched_coverings}
{\rm The map (\ref{eq:p_Y_X}) is said to be a \emph{branched
covering} if both $X$ and $Y$ are  compact Hausdorff spaces, $p$
is a continuous surjective map (in particular, closed) and the
following conditions hold:
\begin{enumerate}
\item
$\widehat{X}(x,U)_j^{(k)}$ is a closed subset of
$\widehat{X}(x,U)_{j+1}^{(k)}$ for any point $x$ of $X$, some its
neighborhood $U$ satisfying (\ref{eq:neiborhoods}) and for all
$k=1,\dots,m$, $j=0,1,2,\dots$,
\item
the cardinalities of the pre-images $p^{-1}(x)$ are uniformly
bounded over $x\in X$.
    \end{enumerate}
    }
\end{dfn}

A finite-fold covering $p : Y\rightarrow X$ of connected compact
spaces, obviously, satisfies the conditions (i), (ii) of
Definition \ref{dfn:branched_coverings}, so it is a particular
case of a branched covering.

\begin{prop}\label{prop:strata}
 Let $p : Y\rightarrow X$ be a branched covering. Then the stratum
 $\widehat{X}_j$ is closed in the next stratum $\widehat{X}_{j+1}$
 for all $j\ge 0$.
\end{prop}

\begin{proof} Consider any point $x\in X$
and its regular neighborhood $U$ satisfying
(\ref{eq:neiborhoods}). Then for any $j$ the set
$\widehat{X}_j\cap U$ coincides with the following finite
intersection of the sets
$$
\widehat{X}(x,U)_{j_1}^{(1)}\cap\dots\cap\widehat{X}
(x,U)_{{j_m}}^{(m)}
$$
over $j_1+\dots +j_m \le j$. In particular, the set
$\widehat{X}_j\cap U$ is closed in $\widehat{X}_{j+1}\cap U$.
Since $X$ is compact, $\widehat{X}_j$ is closed in
$\widehat{X}_{j+1}$ as well.
\end{proof}

\begin{prop}\label{prop:surj_reg_neighb}
 Let $p : Y\rightarrow X$ be a branched covering. Then
for any point $x$ of $X$ there is its regular neighborhood such
that the restriction of $p$ on $V_k$ is surjective for any
$k=1,\dots,m$.
\end{prop}

\begin{proof} Let us suppose  the opposite is true. Then for
some point $x$ of $X$ and for some its regular neighborhood
(\ref{eq:neiborhoods})  one can find a net $\{x_\alpha\}$
converging to $x$ and such that the intersection of its pre-image
with $V_k$ is empty for some $k$. In the other words,
$\{x_\alpha\}$ belongs to $\widehat{X}(x,U)_0^{(k)}$, whereas $x$
lies in $X(x,U)_1^{(k)}$. This contradicts to the condition {\rm
(i)} of Definition \ref{dfn:branched_coverings}.
\end{proof}

\begin{dfn}\label{dfn:locepi}\rm
A map $f:Y\to X$ is said to be a \emph{local epimorphism}
if for any $y\in Y$ and any its neighborhood $U\ni y$
there exists another neighborhood $U_y\ss U$ such that
$ {V_x}:=f(U_y)$ is an (open) neighborhood of $x=f(y)$.
\end{dfn}

\begin{lem}\label{lem:openlocalepim}
A map $f:Y\to X$ is a local epimorphism
if and only if it is an open map.
\end{lem}

\begin{proof}
`If' is evident: take $U_y=U$.

Now let $f$ be a local epimorphism and $U\ss Y$
be an arbitrary open set. For each $y\in U$ find $U_y$
and $V_x$ in accordance with Definition \ref{dfn:locepi}.
Then
$$
f(U)=f\left(\cup_{y\in U} U_y\right)=
\cup_{y\in U}f\left( U_y\right)=\cup_{y\in U}V_{ {f(y)}}
$$
is open.
\end{proof}

\begin{teo}\label{teo:openbranched}
Consider a surjective map $p:Y\to X$ of compact Hausdorff
spaces with uniformly bounded number of pre-images, i.e.
$$
\sup_{x\in X} \# p^{-1}(x)=m<\infty.
$$
Then $f$ is a branched covering if and only if it is open.
\end{teo}

\begin{proof}
 In fact the proof of Proposition \ref{prop:surj_reg_neighb}
may be slightly changed to obtain the `only if' statement. Indeed, for
any point $y\in Y$ with $x=p(y)$ we consider a regular neighborhood
$U$ of $x$ as in Proposition \ref{prop:surj_reg_neighb}. Let $y$ belong to
$V_k$ and $V$ be an arbitrary neighborhood of $y$. Set
$H=p(V\cap V_k)\subset U$, $x\in H$.
Now to make sure that $f$ is a local epimorphism we have
only to verify that $H$ contains some (open) neighborhood of $x$.
But otherwise there is a net $\{x_i\}$ of $U\setminus H$ converging to $x$. Then
the net $y_i:=p^{-1}(x_i)\cap V_k$ converges to $y$ and does
not belong to the (open) neighborhood $V_k\cap V$ of $y$. We come to a
contradiction. Thus $p$ is a local epimorphism and by Lemma
\ref{lem:openlocalepim} it is open.

Now let $p$ be open, hence be a local epimorphism. Suppose,
the first item of Definition \ref{dfn:branched_coverings}
does not hold. This means that for some point $x'\in U$
(may be $x'\ne x$), any its regular neighborhood $U'$
and some its pre-image $y'_k$ we can find a point
$x''\in U$ with no pre-image in $V'_k$. Thus, $p$ is
not a local epimorphism. A contradiction.
\end{proof}

\begin{rk}\label{rk:br_cov_ex}{\rm
If we replace the first condition of Definition
\ref{dfn:branched_coverings}  by the condition that
$\widehat{X}_j$ is closed in  $\widehat{X}_{j+1}$  for all $j\ge
0$, then the statement of Proposition \ref{prop:surj_reg_neighb}
will not be true. The corresponding example is given by Figure
\ref{fig:br_cov_ex},
\begin{figure}[ht]
\begin{picture}(90,90)%(95,0)
\put(0,0){\line(1,0){70}} \put(80,5){$X$}
\put(0,40){\line(1,0){40}} \put(0,40){\line(1,1){40}}
\put(40,40){\line(3,2){30}}
 \put(40,40){\line(3,-2){30}}  \put(80,50){$Y$}
\put(30,30){\vector(0,-1){15}} \put(20,20){$p$}
\end{picture}
 \caption{Remark~\ref{rk:br_cov_ex}}\label{fig:br_cov_ex}
\end{figure}
that differs from Figure
\ref{fig:reflex_mod1} by one additional (closed) interval
ending over the branch point. }
\end{rk}

%%%%%%%%%%%%%%%%%%%%%%%%%%%%%%%%%%%%%%%%%%%%%%%%%
\section{Projective and finitely generated Hilbert $C^*$-modules.
Examples}\label{sec:hilbmodcov}
%%%%%%%%%%%%%%%%%%%%%%%%%%%%%%%%%%%%%%%%%%%%%%%%%

For facts on Hilbert $C^*$-modules we refer the reader to
\cite{Lance, MaTroBook, Raeburn-Williams}.  We recall here just
the most important for us statements. For a Hilbert $C^*$-module $M$
over a $C^*$-algebra $A$ (it always is supposed to be unital in
the present paper, unless otherwise is explicitly stated) the
$A$-dual module $M'$ is the module of all bounded $A$-linear maps
from $M$ to $A$. $M$, equipped with an $A$-inner product  $\langle
\cdot,\cdot\rangle$, is called self-dual if the map $\wedge :
M\rightarrow M'$, $x^{\wedge}(\cdot)=\langle x,\cdot\rangle$ is an
isomorphism, and $M$ is called reflexive if the map $ \cdot:
M\rightarrow M''$, $ \dot{x}(f)=f(x)^*$ $(f\in M')$ is an
isomorphism. Unlike the Banach space situation the third dual
$M'''$ for $M$ is always isomorphic to $M'$, whereas the modules
$M$, $M'$ and $M''$ may be pairwise non-isomorphic in particular
situations (cf. \cite{Paschke2}).

$M$ is called \emph{finitely generated} Hilbert $C^*$-module if it
is an $A$-span of a finite system of its vectors, $M$ is called
\emph{finitely generated projective} if it is a direct summand of
$A^n$ for some $n$. It is easy to see that a finitely generated
projective module over a unital $C^*$-algebra is always self-dual.
$M$ is called \emph{countably generated} if it is a norm-closure
of an $A$-span of a countable system of its vectors. Kasparov's
stabilization theorem asserts that any countably generated Hilbert
$A$-module can be represented as a direct orthogonal summand of
the standard module $l_2(A)$  \cite{KaspJO}.

\begin{teo}\label{teo:f.g._proj}
Any finitely generated Hilbert module over a unital $C^*$-algebra
is a projective one.
\end{teo}

\begin{proof} By Kasparov's stabilization theorem
 a finitely generated module is an orthogonal
direct summand of the standard module $l_2(A)$. Therefore it is
projective by \cite[Theorem 1.3]{MF}.
\end{proof}

Now we will prove more statements
about (not) finitely generated and (not) finitely generated
projective modules over commutative $C^*$-algebras.
Some related examples will be used in the sequel.

The next statement is well known.
\begin{lem}\label{lem:cononfg}
Let $X$ be a compact Hausdorff space and $x_0$ be its non-isolated
point. Then the module
$
C(X)_0:=\{f\in C(X) : f(x_0)=0\}
$
is not finitely generated over $C(X)$.
\end{lem}

\begin{proof}
Assume there is a finite number of generators $f_1,\dots,f_s$ of
$C(X)_0$ over $C(X)$ and consider the function
$$
f=|f_1|^{1/2}+\dots +|f_s|^{1/2}.
$$
Obviously it can not vanish on an entire neighborhood of $x_0$.
Under our assumptions
$$
f=g_1\cdot f_1 + \dots +g_s\cdot f_s
$$
for some $g_i\in C(X)$. Suppose $I$ is the subset of the set
$\{1,\dots,s\}$ such that $g_i$ is not the zero function if and
only if $i\in I$. Let us put
\[
m_i=\|g_i\|=\max_{x\in X} |g_i(x)|, \quad i\in I.
\]
There is an open neighborhood $U$ of $x_0$ such that the following
inequalities
\[
|f_i(x)|^{1/2}\le  \frac 1{2\cdot m_i}
\]
hold for any $x\in U$, $i\in I$. Hence
$$
|f(x)|  \le \sum_{i\in I} |f_i(x)|^{1/2}\cdot |f_i(x)|^{1/2} m_i
\le \frac 12 \sum_{i\in I} |f_i(x)|^{1/2} \le \frac 12
\sum_{i=1}^s |f_i(x)|^{1/2}= \frac 12 |f(x)|.
$$
A contradiction.
\end{proof}

Let
$p : Y\rightarrow X
$
be a continuous map of Hausdorff topological
spaces. Then $C(Y)$ is a Banach $C(X)$-module with respect to the
action:
\begin{gather}\label{eq:action}
(f\xi)(y)=f(y)\xi(p(y)), \qquad f\in C(Y), \xi\in C(X).
\end{gather}

\begin{lem}\label{lem:directed_sets}
Let $X'\subset X$ be a directed set $\{x_\a\}$ together with a
unique limit point $x$. Let $Y'\subset Y$ be equal to the union of
directed sets $\{y^0_\a\}$ and $\{y^1_\a\}$ with a common limit
point $y$, and
$$
p(y^0_\a)=p(y^1_\a)=x_\a,\qquad p(y)=x.
$$
Then $C(Y')$ is not a finitely generated module over $C(X')$.
\end{lem}

\begin{proof}
 Consider two $C^*$-subalgebras $C(Y')_1$ and
$C(Y')_0$ of the $C^*$-algebra $C(Y')$, where $C(Y')_1$ consists
of those continuous functions, which are constant on
$\{y^0_\a\}\cup y$, and $C(Y')_0$ consists of those continuous
functions, which are zero on $\{y^1_\a\}\cup y$. Then any
continuous function $f$ on $Y'$ can be represented in a unique way
as the sum $f=f_1+f_0$ of the function $f_1\in C(Y')_1$, which is
equal to $f$ on $\{y^1_\a\}\cup y$, and the function $f_0=f-f_1\in
C(Y')_0$. Hence, $C(Y')=C(Y')_1 \oplus C(Y')_0$. Clearly,
$C(Y')_1$ is isomorphic to $C(X')$, and $C(Y')_0$ is isomorphic to
$C(X')_0$ as Hilbert $C(X')$-modules, where $C(X')_0$ consists of
continuous functions vanishing at $x$. Thus, if $C(Y')$ is
finitely generated, then $C(X')_0$ is finitely generated too. A
contradiction with Lemma \ref{lem:cononfg}.
\end{proof}

\begin{lem}\label{lem:subspaces}
Given the map {\rm (\ref{eq:p_Y_X})}, where $X, Y$ are normal
Hausdorff spaces. Suppose $X'\subset X$ is a closed subset,
$Y'=p^{-1}(X')$, and $C(Y)$ is a finitely generated $C(X)$-module.
Then
\begin{enumerate}[\rm (i)]
    \item $C(Y')$ is a finitely generated $C(X')$-module.
    \item $C(Y'')$ is a finitely generated $C(X')$-module for any closed subset
    $Y''\subset Y'$.
\end{enumerate}
 \end{lem}

\begin{proof}
(i) Consider generators $\{f_1,\dots,f_n\}$ of $C(Y)$ over $C(X)$
and put $f'_i=f_i|_{Y'}$, $i=1,\dots,n$. Then by the Tietze
theorem for any $h'\in C(Y')$ there is $h\in C(Y)$ satisfying
$h|_{Y'}=h'$. Since $h=f_1g_1+\dots+f_ng_n$ for some $g_1,\dots,
g_n\in C(X)$, one has $h'=f_1'g_1'+\dots+f_n'g_n'$, where
$g_i'=g_i|_{X'}$.

(ii) Given a closed subset $Y''\subset Y'$. For any function
$f''\in C(Y'')$ it is possible to construct its extension  $f'\in
C(Y')$, which may be decomposed as
$f'=f_1'\alpha_1+\dots+f_n'\alpha_n$ with $\alpha_i\in C(X')$.
Then $f''=f_1''\alpha_1+\dots+f_n''\alpha_n$ with
$f_i''=f_i'|_{Y''}\in C(Y'')$ as required.
\end{proof}

\begin{lem}\label{lem:subspaces2}
Given the map {\rm (\ref{eq:p_Y_X})}, where $X$ and $Y$ are
compact Hausdorff spaces. Let $X'\subset X$ be a closed subset and
$Y'=p^{-1}(X')$. If $C(Y)$ is a finitely generated projective
$C(X)$-module, then $C(Y')$ is a finitely generated projective
$C(X')$-module.
\end{lem}

\begin{proof} Let $M=C(Y)$, $M'=C(Y')$. Then one has two
$*$-epimorphisms
\[
\varphi : M\rightarrow M'\quad \mathrm{and}\quad \psi :
C(X)\rightarrow C(X')
\]
given by $\varphi(f)=f|_{Y'}$ and $\psi(\alpha)=\alpha|_{X'}$
respectively, and satisfying the conditions
\[
\varphi(f\alpha)=\varphi(f)\psi(\alpha), \quad \varphi(1)=1, \quad
\psi(1)=1,
\]
where $f\in C(Y)$, $\alpha\in C(X)$. There is an injection $i: M
\hookrightarrow C(X)^n$ and a surjection $s:C(X)^n\to M$, such
that $s\circ i=\Id_M$. In particular, $i\circ s\circ i\circ
s=i\circ s=\pi$ for some idempotent $\pi$ on $C(X)^n$. Obviously,
$i$ is also topologically injective, i.e. $\|i(f)\|\ge k\|f\|$ for
a certain $k>0$ and for any $f\in C(Y)$. Define $i': C(Y')
\hookrightarrow C(X')^n$ in the following way. Take $f'\in C(Y')$,
extend it by Tietze's lemma to a continuous function $f$ on $Y$,
apply $i$ and then $\psi^n$. Evidently, the result does not depend
on the choice of extensions, because of modularity and topological
injectivity of $i$. Moreover, $i'$ is a module map and $i'$ is an
injection. To verify the last statement let us take any function
$f'\in C(Y')$ for which $i'(f')=\psi^ni(f)=0$, where $f$ is a
certain continuous extension of $f'$ to $Y$. This implies
$i(f)|_{X'}=0$. Then for any $\varepsilon>0$ there exists an open
neighborhood $U$ of $X'$ such that $\|i(f)|_U\|<\varepsilon$.
Consider a function $0\le\gamma\le 1$ of $C(X)$, which is $1$ on
$X'$ and $0$ outside of $U$. Assume  that $f'\ne 0$, then
$|f'(y)|=C> 0$ for a certain point $y\in Y'$. Consequently,
$\|\gamma f\|\ge|(\gamma f)(y)|=C$. On the other hand, one has
$\|i(\gamma f)\|=\|\gamma i(f)\|<\varepsilon$ due to modularity of
$i$. But it contradicts to topological injectivity of $f$. Define
in a similar way $s':C(X')^n\to C(Y')$: take $(f'_1,\dots,f'_n)$,
extend them by Tietze's lemma to $(f_1,\dots,f_n)$ on $X$, apply
$s$ and then $\varphi$. It is well defined. Varying functions
$f'_i$ and their extensions we obtain all elements of $C(Y)$ as
$(f_1,\dots,f_n)$. This implies surjectivity of $s'$. Evidently,
$s'\circ i'=\Id_{M'}$.
\end{proof}

\begin{lem}\label{lem:sequence}
Let $X=\{x\}\cup\{x_\a\}$, where $x_\a$ is a net, which converges
to the point $x$, $Y=\{y\}$, $p(y)=x$. Then $C(Y)$ is finitely
generated but not projective Banach module over $C(X)$ with
respect to the action {\rm(\ref{eq:action})}.
\end{lem}

\begin{proof} Evidently, $C(Y)$ is finitely (namely, one)
generated over $C(X)$.
If it is finitely generated projective, then
there exists a $C(X)$-valued inner product $\<.,.\>$ on
$C(Y)$. For any $f\ne 0$ on $Y$ and any $x_\a$ consider
a continuous function $\f:X\to [0,1]$, $\f(x)=1$,
$\f(x_\a)=0$. Then $f=\1\cdot f= \f\cdot f$ and
$$
\<f,f\>(x_\a)=
\<\f f,\f f\>(x_\a)=\f(x_\a)\<f,f\>(x_\a)\f(x_\a)=0.
$$
Since $\a$ is an arbitrary index, $\<f,f\>\equiv 0$.
\end{proof}

\begin{ex}\label{ex:cover_of _circle}\rm
Let $X=S^1=[0,1]/\thicksim$ be a circle, which is thought as the
interval $[0,1]$ whose end points are identified, $Y=[0,1]$, and
$p: Y\rightarrow X$ is defined by the formula $p(t)=[t]$, where
$[t]\in [0,1]/\thicksim$ means the equivalence class of $t$. Then
$C(Y)$ is a Banach $C(X)$-module with respect to the
action~(\ref{eq:action}). We claim that this module is finitely
generated but not projective. Indeed, consider the sets
$Y_1=[0,1/3]$, $Y_2=(1/3,2/3)$, $Y_3=[2/3,1]$ and functions
$\psi_1, \psi_2\in C(Y)$ such that $\psi_1|Y_1=1, \psi_1|Y_3=0$,
$\psi_2|Y_1=0, \psi_2|Y_3=1$ and $\psi_1, \psi_2$ are linear on
$Y_2$. Then $\psi_1+\psi_2=1$ and for any $f\in C(Y)$ the equality
$f=f\psi_1+f\psi_2$ takes place. So $\psi_1, \psi_2$ are
generators of $C(Y)$ over $C(X)$ and the module is finitely
generated. By Lemmas \ref{lem:subspaces2} and \ref{lem:sequence}
it is not projective finitely generated.
\end{ex}

\begin{ex}\rm
Let $x_0$ be a point of the circle $S^1$ and $X=S^1\times
\{x_0\}\cup \{x_0\}\times S^1$ be a union of two circles (i.e.
``8''). Let $Y$ be a disjoint union $S^1\sqcup S^1$ and the
natural surjective map $p : Y\rightarrow X$ has one pre-image for
all points except of $x_0$. Then it immediately follows from
Lemma~\ref{lem:subspaces2} and Example~\ref{ex:cover_of _circle}
that $C(Y)$ is a finitely generated but not projective
$C(X)$-module.
\end{ex}

%%%%%%%%%%%%%%%%%%%%%%%%%%%%%%%%%%%%%%%%%%%%%%%%%%%%%%%
\section{Branched coverings and Hilbert C*-modules}\label{sec:branchmodu}
%%%%%%%%%%%%%%%%%%%%%%%%%%%%%%%%%%%%%%%%%%%%%%%%%%%%%%%
We start this section with a couple of observations.

\begin{lem}\label{lem:count_gen_mod}
Consider the map {\rm(\ref{eq:p_Y_X})}, where $X$, $Y$ are
compact and $Y$ has a countable base. Then $C(Y)$ is a countably
generated module over $C(X)$ with respect to the action {\rm
(\ref{eq:action})}.
\end{lem}

\begin{proof}
Under our assumptions the $C^*$-algebra $C(Y)$ is separable
\cite[1.6.9]{RokhlinFuksBook},~\cite[Prop. 1.11]{ElemNCG}, so it
is a countably generated module over $C(X)$.
\end{proof}

Now we would like to describe an example of a countably, but not
finitely generated Hilbert $C^*$-module arising from the simplest
branched covering. In addition, this example illustrates some
ideas of the proof of Theorem
\ref{teo:branched_coverings_inner_product}.

\begin{ex}\rm\label{ex:reflex_mod1}
Consider the map $p : Y\rightarrow X$ of
Figure~\ref{fig:reflex_mod1},
\begin{figure}[ht]
\begin{picture}(90,60)%(95,0)
\put(0,0){\line(1,0){70}} \put(80,0){$X$}
\put(0,40){\line(1,0){40}} \put(40,40){\line(3,2){30}}
 \put(40,40){\line(3,-2){30}}  \put(80,30){$Y$}
\put(30,30){\vector(0,-1){15}} \put(20,20){$p$}
\end{picture}
 \caption{Example~\ref{ex:reflex_mod1}}\label{fig:reflex_mod1}
\end{figure}
where X is an interval, say [0,1],
and Y is the topological union of one interval with two copies of
another half-interval with a branch point at $1/2$. Then $C(Y)$ is
a Banach $C(X)$-module for the action {\rm (\ref{eq:action})}.
Define a $C(X)$-valued inner product on $C(Y)$ by the formula
\begin{gather}\label{eq:inner_pr}
    \langle f,g\rangle (x)=\frac{1}{\# p^{-1}(x)} \sum_{y\in
    p^{-1}(x)} \overline{f(y)}g(y),
\end{gather}
where $\# p^{-1}(x)$ is the cardinality of the pre-image
$p^{-1}(x)$. The obvious inequality
\begin{gather*}
    \frac{\|f\|^2}{2}\le \|\langle f,f\rangle\|\le\|f\|^2,\quad
    f\in C(Y)
\end{gather*}
implies that the $C^*$-Hilbert norm $\|\langle f,f\rangle\|$ is
 equivalent to the $C^*$-norm on $C(Y)$. Therefore $C(Y)$ is
a Hilbert $C(X)$-module with respect to the inner product
(\ref{eq:inner_pr}) and this module is countably
generated by Lemma~\ref{lem:count_gen_mod}. Moreover, this module
is reflexive by \cite[Theorem 4.4.2]{MaTroBook}.
But this module is not self-dual. Indeed, by
Lemmas~\ref{lem:directed_sets},~\ref{lem:subspaces} it is not a
finitely generated projective one. Recall (cf.~\cite{TroPAMS})
that a unital $C^*$-algebra is said to be MI (module infinite) if
each countably generated Hilbert module over it is projective
finitely generated if and only if it is self-dual.
The $C^*$-algebra
$C(X)$ of this example
is MI by~\cite[Theorem 33]{TroPAMS}, therefore $C(Y)$ is
not a self-dual module over it.
\end{ex}

\begin{teo}\label{teo:branched_coverings_inner_product}
Let $p: Y\rightarrow X$ be a branched covering.
 Then $C(Y)$ may be equipped with a $C(X)$-valued inner product
 in such a way that it becomes a $C(X)$-Hilbert module, whose norm is
 equivalent to the $C^*$-norm of  $C(Y)$.
\end{teo}

\begin{proof}
Given any functions $f$, $g$ of $C(Y)$. We will construct their
$C(X)$-valued inner product by induction over the sets
$\widehat{X}_j$, $j=0,1,\dots, N$. Suppose ${X}_{j_1}$ is
the first non-empty stratum. Then the formula
\begin{gather}\label{eq:base_induct}
\langle f,g\rangle(x)=\frac{1}{\#p^{-1}(x)}
    \sum_{y\in p^{-1}(x)} \overline{f(y)} g(y)
\end{gather}
provides the base of induction. Now suppose the inner product is
defined on the strata $X_1,\dots,{X_j}$ and the next non-empty set
is $X_{j+k}$, $k>0$.

 By Proposition
\ref{prop:surj_reg_neighb} for any point
$x\in \widehat{X}_j$ there exists its regular neighborhood $U$
satisfying (\ref{eq:neiborhoods}) such that the restriction of $p$
on $V_k$ is surjective for any $k=1,\dots,m$. We will define the
inner product $\langle f,g\rangle$ at any point $z$ of $U\cap
{X}_{j+k}$ as follows. Let
$$
p^{-1}(z)\cap V_k=\{u_1^{(k)},\dots,u_{i_k}^{(k)}\},
$$
where $i_1+\dots+i_m=j+k$ and  $i_k\neq 0$ for any $0$. Denote
$$
    f_{k}:=f|_{V_{k}},\quad g_{k}:=g|_{V_{k}}
$$
and define  a function $\langle f_{k},g_{k}\rangle : U\cap
{X}_{j+k} \rightarrow \mathbb{C}$ by the formula:
\begin{equation}\label{eq:def_inn_prod}
\langle f_{k},g_{k}\rangle (z)= \frac{1}{i_k}\sum_{t=1}^{i_k}
\overline{f_{k}(u_{t}^{(k)})}g_{k}(u_{t}^{(k)}).
\end{equation}
Then
\begin{equation}\label{eq:def_inn_prod2}
\langle f,g\rangle_{U\cap {X}_{j+k}} (z)= \frac{1}{m}
    \sum_{k=1}^{m}\langle f_{k},g_{k}\rangle (z).
\end{equation}

Consider such a regular
neighborhood $U=U(x)$ for each point $x\in \widehat{X}_j$.
Extend the system $\{U(x) : x\in
\widehat{X}_j\}$ up to a cover of $\widehat{X}_{j+k}$ by open
sets $O_i$ satisfying $O_i\cap \widehat{X}_j= \varnothing$. Let
$\{U_1,\dots,U_K, O_1,\dots, O_M\}=\{W_1,\dots,W_{K+M}\}$ be a
finite subcovering of the compact space $\widehat{X}_{j+k}$ and
$\{\varphi_i(x)\}_{i=1}^{K+M}$ be a partition of unity
subordinated to this
subcovering. Define $\langle f,g\rangle_{W_i}$ over $W_i$ by the
formulas (\ref{eq:def_inn_prod}), (\ref{eq:def_inn_prod2}) if
$i\le K$ and by the formula (\ref{eq:base_induct}) otherwise.
Define an inner product on
$C(p^{-1}(\widehat{X}_{j+k}))$ in the following way:
\begin{equation}\label{eq:inner_prod_reflth}
    \langle f,g\rangle(x)=\sum_{i=1}^{K+M}\langle f,g\rangle_{W_i} (x)
    \varphi_i(x),
\end{equation}
where $f,g\in C(Y)$, $x\in \widehat{X}_{j+k}$. The inductive step
is complete.

 We claim that $\langle f,g\rangle$ is continuous on $X$.
 Indeed,  consider  any point $x\in X$
 and any net $\{x_\alpha\}$ converging to $x$.
Then $x\in X_j$ for some $j$. Denote $\{x_\alpha^{(i)}\}=
\{x_\alpha\}\cap X_i$. By Proposition \ref{prop:strata} we can
assume that $i\ge j$. It remains to verify that for any $i$
the difference $|\langle f,g\rangle (x)-\langle f,g\rangle
(x_\alpha^{(i)})|$ goes to zero when $x_\alpha^{(i)}$ goes to $x$.
But it directly follows from the definition of the inner product,
namely from the continuity of (\ref{eq:def_inn_prod}).

Thus $\<f,f\>(x)$ is a convex combination of not more than
$N=\max\limits_{x\in X} \#p^{-1}(x)<\infty$ numbers $|f(y_i)|^2$,
where $p(y_i)=x$. Hence we obtain the following inequality
\begin{gather*}
    \frac{\|f\|^2}{N}\le \|\langle f,f\rangle\|\le\|f\|^2,\quad
    f\in C(Y).
\end{gather*}
Thus the Hilbert norm $\|\langle f,f\rangle\|$ is equivalent to
the $C^*$-norm of $C(Y)$.
\end{proof}

\begin{teo}\label{teo:criterium_fin_gen_mod2}
Suppose $X$ and $Y$ are  compact Hausdorff connected spaces and
$p:Y\rightarrow X$ is a continuous surjection. If $C(Y)$ is a
projective finitely generated Hilbert module over $C(X)$ with
respect to the action~{\rm(\ref{eq:action})}, then $p$ is a
finite-fold covering.
\end{teo}

\begin{proof}
Let functions $g_1,\dots, g_n$ generate the projective module
$C(Y)$ over $C(X)$. Then we claim that the cardinality of the
pre-image of any point $x\in X$ does not exceed $n$. Indeed,
assume there is a point $x\in X$, whose pre-image is $\{y_1,
\ldots ,y_m\}$ and $m>n$. By the Urysohn's lemma there are
continuous functions $f_1,\ldots ,f_m\in C(Y)$ such that
$f_i(y_i)=1$ and $f_i(y_j)=0$ whenever $i\neq j$. The functions
$f_1,\ldots ,f_m$ can be expressed as linear combinations of the
generators $g_1,\dots, g_n$ with coefficients from $C(X)$. Let us
denote by $\widehat{f}_i$ and $\widehat{g}_j$ the restrictions of
$f_i$ and $g_j$ onto $\{y_1, \ldots ,y_m\}$. Then both
$\widehat{f}_i$ and $\widehat{g}_j$ belong to the vector space
\begin{gather*}
C(\{y_1, \ldots ,y_m\})\cong \mathbb{C}^m.
\end{gather*}
The vectors $\widehat{f}_1,\dots,\widehat{f}_m$
form a base of this vector space and, consequently, they can not
be represented as linear combinations of the vectors
$\widehat{g}_1,\dots,\widehat{g}_n$, when $m>n$. Thus,
$m$ does not exceed $n$.

Assume $k_x$ denotes the cardinality of the pre-image of a point
$x\in X$ and $k$ is a minimal value of $k_x$'s over $x\in X$.
Firstly, we claim that the set $X_k=\{x\in X : k_x=k\}$ is open.
Indeed, in the opposite case
 there is a net $\{x_\alpha\}$ in $X\setminus X_k$
 converging to a certain point $x$ of $X_k$. By Lemma
 \ref{lem:neiborhoods} one can  found a regular neighborhood $U$ of $x$
 satisfying the condition (\ref{eq:neiborhoods}) with $m=k$.
 Moreover, one can assume
 (passing to a sub-net of $\{x_\alpha\}$ if it is necessary)
that the net $\{x_\alpha\}$ belongs to $U$ and there is a number
$i$ such that the neighborhood $V_i$ has at least two points
$y'_\alpha$ and $y''_\alpha$ from the pre-image of $x_\alpha$ for
any $\alpha$. Put $X'=\{x\}\cup\{x_\alpha\}$ and $Y'=\{y\}\cup\{
y'_\alpha\}\cup\{y''_\alpha\}$, where $y=p^{-1}(x)\cap V_i$. Then
$C(Y')$ is a finitely generated module over $C(X')$ by
Lemma~\ref{lem:subspaces}. But this contradicts to
Lemma~\ref{lem:directed_sets}.

Secondly, let us show that $X_k$ is closed. In the opposite case
there is a net $\{x_\alpha\}$ of $X_k$ converging to some point
$x$ of $X_j$ with $j>k$. Denote $X'=\{x\}\cup\{x_\alpha\}$,
$Y'=p^{-1}(X')$ and choose neighborhoods $U$, $U'$ of the point
$x$ and $V_i$, $V_i'$ $(i=1,\dots,j)$ as in
Lemmas~\ref{lem:neiborhoods},~\ref{lem:neiborhoods2}. Then
$C(\sqcup \overline{V_i'})=\oplus C(\overline{V_i'})$ is a
finitely generated projective $C(\overline{U'})$-module by Lemma
\ref{lem:subspaces2}. Therefore, obviously, each
$C(\overline{U'})$-module $C(\overline{V_i'})$ is finitely
generated  too. We can assume (passing to a sub-net of
$\{x_\alpha\}$ if it is necessarily) that the intersection of the
set $p^{-1}(\{x_\alpha\})$ with a neighborhood $V_i$ is empty for
some number $i$.  Now consider the submodule
$C(p^{-1}|_{\overline{V_i'}}(X'))=C(\{y_i\})$ of the module
$C(\overline{V_i'})$, where $y_i=p^{-1}(x)\cap V_i$.
It has to be
finitely generated projective by
Lemmas~\ref{lem:subspaces} and \ref{lem:subspaces2}, but it is
impossible by Lemma~\ref{lem:sequence}.

So we have proved that the set $X\setminus X_k$ is both open and
closed and, consequently, it has to be empty, because $X$ is
supposed to be connected. Thus, all points of $X$ have the same
number of pre-images.

Now for an arbitrary point $x\in X$ let us choose its regular
neighborhood $U$ satisfying the condition~(\ref{eq:neiborhoods})
with $m=k$. Then $p$ is a (local) bijection, which is closed and
open (by our argument for branched coverings). Thus it is
a local homeomorphism.
\end{proof}

We complete this section with a couple of statements
relating coverings to some other classes of Hilbert
$C^*$-modules.

\begin{teo}\label{teo:self-dual_criterium}
Consider the map {\rm (\ref{eq:p_Y_X})}, where
 $X$ and $Y$ are compact spaces, $X$ is connected and
 $Y$ has a countable
base. Then the following conditions are equivalent:
\begin{enumerate}
    \item $C(Y)$ is a self-dual module with respect to the
    action {\rm(\ref{eq:action})};
    \item the map {\rm (\ref{eq:p_Y_X})} is a finite-fold
    covering.
\end{enumerate}
\end{teo}

\begin{proof}
The implication ${\rm (ii)}\Rightarrow {\rm (i)}$
follows from \cite[Proposition 2.8.9]{Wata}
and  \cite[pp. 92--93]{Wata}
(see also Theorem \ref{teo:cond_expect} below).
To prove the
inverse implication let us remark that $X$ does not have
isolated points because it is connected, so the $C^*$-algebra
$C(X)$ is MI by~\cite[Theorem 33]{TroPAMS}. According to our
assumptions and Lemma \ref{lem:count_gen_mod} the $C(X)$-module
$C(Y)$ has to be both countably generated and self-dual. Therefore
it is a finitely generated projective module. Then by
Theorem \ref{teo:criterium_fin_gen_mod2} the map
(\ref{eq:p_Y_X}) defines a finite-fold covering.
\end{proof}

\begin{teo}\label{teo:refl_branched_cov}
Let $p: Y\rightarrow X$ be a branched covering over a compact
metric space $X$.  Then the $C(X)$-Hilbert module $C(Y)$ is
$C(X)$-reflexive.
\end{teo}

\begin{proof} It follows from Theorem
\ref{teo:branched_coverings_inner_product}
and \cite[Theorem 4.1]{FMT09}.
\end{proof}

%%%%%%%%%%%%%%%%%%%%%%%%%%%%%
\section{Branched coverings and conditional expectations}
\label{sec:branchcovexpect}
%%%%%%%%%%%%%%%%%%%%%%%%%%%%%%

In this section we complete proofs of Theorems \ref{teo:mainbranchcov}
and \ref{teo:mainhcover} working with conditional
 expectations. Recall briefly some necessary facts from
 \cite{Wata}
 (see also \cite{Ped, Takesaki1}).

\begin{dfn}\label{dfn:condexp}\rm
 Suppose, $B$ is
 a $C^*$-algebra  and $i:A\hookrightarrow B$ is its $C^*$-subalgebra.
 A \emph{conditional expectation}
 $E: B\rightarrow A$ is a surjective projection of norm one satisfying
 the following conditions:
 \[
E(i(a)\cdot b)=aE(b),\quad E(b\cdot i(a))=E(b)a,\quad E(i(a))=a,
 \]
 for $a\in A$, $b\in B$. We deal with unital $C^*$-algebras
 and we will always assume, that
 \begin{enumerate}
    \item $E$ is \emph{positive}: $E(b^*b)\ge 0$ for any $b\in B$;
    \item $E$ is \emph{unital}, i.e. $i$ is unital, or $A$ and $B$
    have a common unity.
 \end{enumerate}
\end{dfn}

\begin{dfn}\label{dfn:condexpalgfi}\rm
 A family $\{u_1,\dots,u_n\}\subset B$
 is called a \emph{quasi-basis} for $E$ if
 \[
b=\sum_j u_j E(u_j^*b)\qquad \text{for } b\in B.
 \]
 A conditional expectation $E: B\rightarrow A$
 is \emph{algebraically of index-finite
 type} if there exists a finite quasi-basis for $E$. In this case the
 index of $E$ is defined by: $\mathrm{Index} (E)=\sum_j u_j u_j^*$,
 which is a positive invertible element in the center of $B$ and
 it does not depend on the choice of the quasi-basis
 $\{u_1,\dots,u_n\}$.
\end{dfn}

\begin{dfn}\label{dfn:exptopolfiniteindex}\rm
Given a $C^*$-algebra $B$ and its $C^*$-subalgebra $A$. A
conditional expectation $E: B\rightarrow A$ is \emph{topologically
of index-finite type} \cite{BDH} (see also \cite{FrankKirJOP}) if
the mapping $(K\cdot E-{\rm id}_B)$ is positive for some real
number $K\ge 1$.
\end{dfn}

We need the following result \cite[Theorem 1]{FrankKirJOP} (see
also \cite[Proposition 3.3]{BDH},
\cite[Theorems 3.4, 3.5]{Jolis91},
\cite[Proposition 2.1, Corollary 2.4]{AnStoj94},
\cite[Theorem 1.1.6, Remark 1.1.7]{Popa1995},
\cite[Proposition 1.1]{FMTZAA}):

\begin{prop}\label{prop:FMTZAA}
 Let $E: B\rightarrow A$ be a conditional expectation.
Then the following conditions are equivalent:
\begin{enumerate}
    \item $E$ is topologically of index-finite type;
    \item $E$ is faithful and the  pre-Hilbert $A$-module
    $\{B,E(\langle \cdot,\cdot\rangle_B)\}$ is complete with respect
    to the norm $\|E(\langle \cdot,\cdot\rangle_B)\|_A^{1/2}$.
\end{enumerate}
\end{prop}

\begin{proof}
The second condition means that the original norm and the Hilbert
module norm on $B$ are equivalent, in particular,
$$
K\|E(x^*x)\|\ge \|x^*x\|
$$
for some constant $K>0$ and for any $x\in B$. Consider an element
$x=b (\e+ E(b^*b))^{-\frac 12}$ for $\e>0$. Then, obviously,
\[
(\e+ E(b^*b))^{-\frac 12} E(b^*b) (\e+ E(b^*b))^{-\frac 12}\le
1_B,
\]
what exactly means that $E(x^*x)\le 1_B$. Hence,  $\|x^*x\|\le K$,
or, equivalently, $x^*x\le K\cdot 1_B$.  In other words, one has
$$
(\e+ E(b^*b))^{-\frac 12} b^*b (\e+ E(b^*b))^{-\frac 12}\le K\cdot
1_B,
$$
which may be rewritten as $K(\e+ E(b^*b))\ge b^*b$. Since $\e>0$
is arbitrary, we are done. The converse is immediate.
\end{proof}

Let $i:A\to B$ be a unital inclusion of commutative
C*-algebras $A=C(X)$, $B=C(Y)$. Then its Gelfand dual
$p=i^*:Y\to X$ is an epimorphism.

\begin{dfn}\label{dfn:c_e_point-wise_comm_case}{\rm
A conditional expectation $E: C(Y)\rightarrow C(X)$ is said to be
\emph{fiber-wise} if for any $x\in X$ and $f\in C(Y)$ such that
$f|_{p^{-1}(x)}=0$ one has $E(\overline{f}f)(x)=0$, $p=i^*$,
$i:C(X)\ss C(Y)$.
 }
 \end{dfn}

Any unital $E$ is fiber-wise. Indeed, up to re-denoting
it is sufficient to prove that $E(f)(x)=0$. For any $\e>0$ there
is a neighborhood $U_\e$ of $p^{-1}(x)$ such that $|f(y)|<\e$
for any $y\in U_\e$. Choose a neighborhood $V_\e$ of $x$
such that $p^{-1}(V_\e)\ss U_\e$ and $a_\e\in C(X)=A$ such that
$\|a_\e\|=1$, $a_\e(x)=0$, $a_\e(x')=1$ for any $x'\not\in V_\e$.
Then $\|E(f-a_\e f)\|\le \|f-a_\e f\|<\e$ and
$E(a_\e f)(x)=a_\e(x) E(f)(x)=0$. Since $\e$ is arbitrary, we
are done.

Note, that the conditional expectation related to the inner
 product constructed in Theorem
 \ref{teo:branched_coverings_inner_product}
 is, obviously, fiber-wise.

\begin{teo}\label{teo:branched_cov_cond_expectations}
Let $X$, $Y$ be compact spaces, $i: C(X)\rightarrow C(Y)$ be a
unital
$*$-inclusion of $C^*$-algebras and $E: C(Y)\rightarrow C(X)$ be a
$($unital positive$)$ conditional expectation
topologically of index-finite
type. Then the map $p=i^*: Y\rightarrow X$ is a branched covering.
\end{teo}

\begin{proof} The map $p$ is surjective and continuous.
The number of pre-images of $p$ is uniformly bounded over
$X$. Indeed, suppose a point $x\in X$ has $n$ pre-images
$\{y_1,\dots,y_n\}$.
Consider non-negative
functions $f_k\in C(Y)$, $k=1,\dots,n$
such that $f_k(y_j)=\delta_{kj}$, $f_k:Y\to [0,1]$
and $f_k f_j$=0 if $k\neq j$.
By Definition \ref{dfn:exptopolfiniteindex}
$(i E)(f_k)(y_k)>\frac 1K f_k(y_k)$, i.e. $E(f_k)(x)>\frac 1K$.
Let $\1\in C(Y)$ be the unity element.
Then by positivity of $E$ we have
$$
E(\1)(x)\ge E\left(\sum_{k=1}^n f_k\right)(x)=
\sum_{k=1}^n E\left(f_k\right)(x)>\frac nK.
$$
Thus, if $n$ is not bounded, then $E(\1)=\1$ is not bounded.
A contradiction.

Now let us verify the item (i) of Definition
\ref{dfn:branched_coverings}. By Theorem \ref{teo:openbranched}
it is sufficient to verify that $p$ is an open map.
Suppose that $p$ is not open, i.e. there is an open set $V\ss Y$
such that $p(V)\ss X$ is not open. Let $x\in p(V)$ be a limit
point of $X\setminus p(V)$ and $y \in V$ its pre-image.
Consider a positive function $f:Y\to [0,1]$ such that $f(y)=1$
and $f$ vanishes outside $V$. Then $E(f)(x)>1/K$ while
$E(f)$ vanishes on $X\setminus p(V)$. Thus $E(f)$ is not
continuous. A contradiction.
\end{proof}

This completes the proof of Theorem \ref{teo:mainbranchcov}.
The next statement completes the proof of Theorem
\ref{teo:mainhcover}.

\begin{teo} \label{teo:cond_expect}
Let $E: B\rightarrow A$ be a conditional expectation,
where $C^*$-algebras $A$ and $B$ have a common unity.
Then the following conditions are equivalent:
\begin{enumerate}
    \item $E$ is algebraically of index-finite type;
    \item $B$ is a finitely generated projective $A$-module.
\end{enumerate}
\end{teo}

\begin{proof} The statement \cite[Corollary 3.1.4]{Wata}
 differs from our theorem
only by one additional condition, which may be omitted in the
unital $C^*$-case, because in this situation
all finitely generated projective modules are self-dual.
\end{proof}

The remaining part of the section is devoted to a clarifying of
the role of the index of $E$
in our theory.  The main idea of \cite{BDH}
concerning the definition of the index element is the following.
Given a $W^*$-algebra $B$ and its $W^*$-subalgebra $A$. Consider a
conditional expectation $E: B\rightarrow A$ topologically of
index-finite type, then it defines an $A$-Hilbert module structure
on $B$. Choose any quasi-orthonormal basis $\{x_i\}$ (relating to
this inner structure) in $B$ and define the index of $E$ as the
sum $\sum x_i^*x_i$ with respect to the ultra-weak topology.
Actually, this definition is very close to the frame approach
elaborated in \cite{FrankLar, FrankLarJOP}. The index of $E$
provided both $B$ and $A$ are $C^*$-algebras was defined in
\cite{FrankKirJOP}, but in this situation it is an element of the
enveloping von Neumann algebra $B^{**}$ of $B$.

We have constructed in the proof of Theorem
\ref{teo:branched_coverings_inner_product}
a function $\mu:Y\to [0,1]$, such that $\sum_{p(y)=x} \mu(y)=1$
for any $x\in X$. This function (not uniquely determined !)
was used to define a $C(X)$-valued inner product $\<.,.\>=\<.,.\>_\mu$
in such a way that
$$
\<f,f\>_\mu(x)=\sum_{p(y)=x} f^*(y)f(y)\mu(y).
$$
Similarly for the induced conditional expectation $E=E_\mu$:
$$
E_\mu (f) (x) =\sum_{p(y)=x} f(y)\mu(y).
$$
As we have explained, this expectation $E_\mu$ is topologically of
index-finite type. Thus, by \cite{FrankKirJOP}, its index element
$\Index(E_\mu)\in B^{**}$ is defined, valued in the enveloping von
Neumann algebra of $B=C(Y)$.

In the remaining part of the section all spaces
are supposed to be \emph{second countable}.

Choose a countable partition of $X_j$,
$$
X_j=X_j^1 \sqcup \dots \sqcup X_j^{r} \sqcup \dots
$$
 in such a way that $X_j^s$ is open in
 $X_j^{s} \sqcup X_j^{s+1}\sqcup \dots $
 and $X_j^s$ is inside of some regular neighborhood
 of a point of $X_j$.
 For this purpose we take a countable covering $U_1,U_2,\dots$
 of $X_j$ with regular neighborhoods centered
 in points of $X_j$ and take
 $$
X_j^1:=U_1\cap X_j,\quad X_j^2:=(U_2\setminus U_1)\cap X_j,
\quad \dots \quad
X_j^{r}:=(U_r\setminus (U_1\cup\dots\cup U_{r-1}))\cap X_j,
\quad \dots
 $$
Let $p^{-1}(X_j^s)=Y_j^{s,1}\sqcup \dots \sqcup Y_j^{s,j}$ be its
``regular'' decomposition. Thus, $Y$ is a disjoint union of
countably many Borel sets $Y_j^{s,t}$. Define
$$
m_{jkt} (y)=\left\{%
\begin{array}{ll}
    \frac 1{\sqrt {\mu(y)}}, & \hbox{$y\in Y_j^{k,t}$ ;} \\
    0, & \hbox{otherwise.} \\
\end{array}%
\right.
$$
Then $m_{jkt}$ are pairwise orthogonal, $\<m_{jkt},m_{jkt}\>_\mu
(x)=0$, if $x\not\in X_j^k$, and
$$
\<m_{jkt},m_{jkt}\>_\mu (x)=\sum_{p(y)=x} m_{jkt}^*(y) m_{jkt}(y)\mu (y)
  =1
$$
if $x\in X_j^k$. Let $M:=\sum_{j,k,t} m_{jkt}^* m_{jkt}$. It is a
bounded function, which is continuous (being $\frac 1\mu$) on each
component of the disjoint union
$$
Y=\bigsqcup Y_j^{k,t}
$$
of Borel sets. Also, it is bounded (by the maximal number
of pre-images under $p$).

\begin{teo}\label{teo:vesiindex}
In this situation
$$
M=\Index(E_\mu).
$$
\end{teo}

\begin{proof}
In fact (cf. \cite{BDH}, \cite{Jolis91}, \cite{FrankKirJOP})
it is sufficient to verify that for any $y\in Y$
and any $f\in C(Y)$
$$
f(y)=\sum_{j,k,t} m_{jkt} (i\, E_\mu) (m^*_{jkt} f)(y)
$$
(in our notation with the inclusion $i$). We have
$y\in Y_l^{p,s}$ for some (uniquely defined) indices $l$, $p$,
and $s$. Then
\begin{eqnarray*}
  \sum_{j,k,t} m_{jkt} (i\, E_\mu) (m^*_{jkt} f)(y)
  &=& m_{lps}(y) \sum_{p(y')=p(y)} m^*_{lps}(y') f(y') \mu(y') \\
    &=& m_{lps}(y) m^*_{lps}(y) f(y) \mu(y)= f(y).
\end{eqnarray*}
\end{proof}

\begin{rk}\rm
The equality in Theorem \ref{teo:vesiindex}
should be considered in $B^{**}$.
In particular, $M$ is an element of $B^{**}$
in the following sense. We approximate $M$
(in fact each $m^*_i m_i$) by a sequence
of continuous functions point-wise. In
fact it is sufficient to approximate the
characteristic function of a set of the
form $K\setminus K'$, where $K$ and $K'\ss K$
are compacts. For each of them it is well
known how to find such a sequence, and then
we take the difference. Finally, we apply
the Egoroff theorem (see e.g. \cite[Sect. 21]{Halmos})
to see that the sequence converges ultraweakly,
i.e. the values on each regular positive measure
converge. See also \cite[III.1 and III.2]{Takesaki1}.
\end{rk}

Let us illustrate the above general considerations by the
following simple example (see \cite[Example 3.3]{FrankKirJOP} for
a similar situation).

\begin{ex}\rm
Consider the branched covering of Example \ref{ex:reflex_mod1}
defined by Figure \ref{fig:reflex_mod1}. This covering is equipped
with a conditional expectation $E: C(Y)\rightarrow C(X)$ of
index-finite type given by the formula
\[
E(f)(x)=\frac{1}{\# p^{-1}(x)}\sum_{y\in p^{-1}(x)} f(y)
\]
and the inner product (\ref{eq:inner_pr}) satisfies $\langle
f,g\rangle:=\langle f,g\rangle_E=E(f^*g).$ In this example the
weight function $\mu$ is $1$ over $[0,1/2]$ and $1/2$ over
$(1/2,1]$. Now we enumerate three intervals forming $Y$ in
the following way: by $1$ for the horizontal interval
and by $2$ and $3$ for two
others. Define three functions $e_1, e_2, e_3$ of $C(Y)$ such that
$e_1$ is equal to $1$ over the first interval of $Y$ and $0$
otherwise, and $e_i$ equals to $\sqrt{2}$ over the $i$-th interval
and $0$ otherwise for $i=2, 3$. Then obviously the vectors
$\{e_i\}$ form a quasi-orthonormal basis (i.e. an orthogonal system,
where inner squares of all vectors are projections) of the
$C(X)$-Hilbert module $(C(Y)^{**},\langle \cdot,\cdot\rangle_E)$.
Thus the index $\Index(E)$ coincides with the sum $\sum e_i^*e_i$.
This function is equal to $1$ over the first subinterval of $Y$,
and to $2$ over two other subintervals of $Y$. Its value in the
branching point defines an element of the discrete part of
$C(Y)^{**}$.
\end{ex}

\begin{ex}\label{ex:two_circles}\rm
Let $X$ be a unit circle and $Y$ consists of two disjoint copies
of $X$, in which the zero-point below is connected by an interval
with the $\pi/2$-point above as it is shown by Figure
\ref{fig:two_circles}.
\begin{figure}[ht]
\begin{picture}(120,120)%(95,0)
\put(50,110){\oval(100,20)} \put(50,60){\oval(100,20)}
\put(50,10){\oval(100,20)} \put(15,50){\line(5,4){62}}
\put(110,110){$Y$}
  \put(110,10){$X$}
  \put(15,50){\circle*{3}}
  \put(77,99){\circle*{3}}
\put(50,45){\vector(0,-1){20}} \put(60,35){$p$}
\end{picture}
 \caption{Example~\ref{ex:two_circles}}\label{fig:two_circles}
\end{figure}
Obviously, $p$ is open and by Theorem
\ref{teo:openbranched} it is a branched covering. The weight
function $\mu$, constructed by the formulas
$(\ref{eq:base_induct})-(\ref{eq:inner_prod_reflth})$ of Theorem
\ref{teo:branched_coverings_inner_product}, is equal to 1/2 on
fibers of cardinality 2, is equal identically to 1/4 on the
line between the circles, is the function
$f_1(x):=\frac{x}{2\pi}+\frac14$ over the interval $(0,\pi/2)$ of
the circle below, and is the function
$f_2(x):=-\frac{x}{2\pi}+\frac12$ over the interval $(0,\pi/2)$ of
the circle above. Then the index element of the corresponding
conditional expectation is a function of $C(Y)^{**}$, which equals
$\sqrt{2}$ on fibers, whose cardinality is 2, equals  2 on
the line between the circles, is the function $1/\sqrt{f_1(x)}$
over the interval $(0,\pi/2)$ of the circle below, and is the
function $1/\sqrt{f_2(x)}$ over the interval $(0,\pi/2)$ of the
circle above.
\end{ex}

\def\cprime{$'$}
\providecommand{\bysame}{\leavevmode\hbox to3em{\hrulefill}\thinspace}
\providecommand{\MR}{\relax\ifhmode\unskip\space\fi MR }
\providecommand{\MRhref}[2]{%
  \href{http://www.ams.org/mathscinet-getitem?mr=#1}{#2}
}
\providecommand{\href}[2]{#2}

\end{document}